\pdfoutput=1
\RequirePackage{ifpdf}
\ifpdf 
\documentclass[pdftex]{sigma}
\else
\documentclass{sigma}
\fi

\begin{document}

\allowdisplaybreaks

\renewcommand{\thefootnote}{$\star$}

\newcommand{\arXivNumber}{1510.01901}

\renewcommand{\PaperNumber}{101}

\FirstPageHeading

\ShortArticleName{Hankel Determinants of Zeta Values}

\ArticleName{Hankel Determinants of Zeta Values\footnote{This paper is a~contribution to the Special Issue
on Orthogonal Polynomials, Special Functions and Applications.
The full collection is available at \href{http://www.emis.de/journals/SIGMA/OPSFA2015.html}{http://www.emis.de/journals/SIGMA/OPSFA2015.html}}}

\Author{Alan {HAYNES}~$^\dag$ and Wadim {ZUDILIN}~$^\ddag$}

\AuthorNameForHeading{A.~Haynes and W.~Zudilin}

\Address{$^\dag$~Department of Mathematics, University of York,\\ \phantom{$^\dag$}~York, YO10\,5DD, UK}
\EmailD{\href{mailto:akh502@york.ac.uk}{akh502@york.ac.uk}}
\URLaddressD{\url{http://www-users.york.ac.uk/~akh502/}}

\Address{$^\ddag$~School of Mathematical and Physical Sciences, The University of Newcastle,\\ \phantom{$^\ddag$}~Callaghan NSW 2308, Australia}
\EmailD{\href{mailto:wzudilin@gmail.com}{wzudilin@gmail.com}}
\URLaddressD{\url{http://wain.mi.ras.ru/}}

\ArticleDates{Received October 19, 2015, in f\/inal form December 16, 2015; Published online December 17, 2015}

\Abstract{We study the asymptotics of Hankel determinants constructed using the values $\zeta(an+b)$ of the Riemann zeta function at positive
integers in an arithmetic progression. Our principal result is a Diophantine application of the asymptotics.}


\Classification{11J72; 11C20; 11M06; 41A60}

\renewcommand{\thefootnote}{\arabic{footnote}}
\setcounter{footnote}{0}

\section{Introduction}
\label{intro}

In the recent work \cite{Mo09}, H.~Monien investigated analytic aspects of the Hankel determinants
\begin{gather*}
H_n^{(r)}=H_n^{(r)}[\zeta]=\det_{1\le i,j\le n}\bigl(\zeta(i+j+r)\bigr)
\qquad\text{for}\quad r=0,1,
\end{gather*}
where $\zeta(s)$ denotes the Riemann zeta function. He
also studied more general determinants constructed using values of Dirichlet series. One focus of that work was the asymptotic behaviour
of $H_n^{(0)}$ and $H_n^{(1)}$ as $n\to\infty$, and a heuristic justif\/ication for the simplif\/ied asymptotic formula
\begin{gather*}
\log H_n^{(r)}=-n^2\log n+O\big(n^2\big) \qquad\text{as}\quad n\to\infty\quad\text{for}\quad r\ge 0.
\end{gather*}

In \cite{Mo11} Monien developed these ideas further and rigorously justif\/ied
the above asymptotics in the case $r=0$, by explicitly constructing a family of orthogonal polynomials related to the corresponding
Riemann--Hilbert problem. However, his approach does not readily generalize to prove the expected asymptotics for determinants built
on the zeta values $\zeta(an+b)$ along an arithmetic progression, which are more interesting from an arithmetical point of view.
To be precise, for positive integers $a$ and $b$, we expect that
\begin{gather*}
\log\det_{1\le i,j\le n}\bigl(\zeta(a(i+j)+b)\bigr)=-an^2\log n+O\big(n^2\big) \qquad\text{as}\quad  n\to\infty.
\end{gather*}
This type of result is of interest, for example, when~$a$ is even and~$b$ is odd, so that the obviously irrational (and transcendental)
zeta values at positive even integers are excluded.

In this brief note we demonstrate how elementary means can be used to prove the weaker asymptotic inequality
\begin{gather*}
\log\Bigl|\det_{1\le i,j\le n}\bigl(\zeta(a(i+j)+b)\bigr)\Bigr|\le-\frac a2 n^2\log n+O\big(n^2\big) \qquad \text{as}\quad  n\to\infty,
\end{gather*}
which leads us to the following arithmetic application.

\begin{theorem}
\label{th-main}
For any pair of positive integers $a$ and $b$, either there are infinitely many $n\in\mathbb{N}$ for which $\zeta(an+b)$ is irrational,
or the sequence $\{q_n\}_{n=1}^\infty$ of common denominators of the rational elements of the set $\{\zeta(a+b),\zeta(2a+b),\dots,\zeta(an+b)\}$
grows super-exponentially, i.e., $q_n^{1/n}\to\infty$ as $n\to\infty$.
\end{theorem}

The current knowledge about the arithmetic of odd zeta values~-- the numbers $\zeta (n)$ for $n\ge 2$ odd~-- can be summarised as follows.
In 1978 R.~Ap\'ery proved~\cite{Ap79} the irrationality of $\zeta(3)$, and in 2000 K.~Ball and T.~Rivoal showed \cite{BR01} that
there are inf\/initely many irrational numbers among the odd zeta values.
At the same time, it is widely believed that all of the numbers $\zeta (n)$, for $n\ge 2$, are irrational and transcendental.
Our result in Theorem~\ref{th-main} is very far from proving this. The goal of this paper is, rather, to demonstrate how very simple analytic arguments
can be used to derive some information about the Diophantine approximation properties of these numbers.

\section{Asymptotic upper bounds}
\label{est}

Suppose that $\{a_n\}_{n=1}^\infty$ is a sequence of complex numbers which satisf\/ies
$a_n\ll n^{1-\delta}$ for some $\delta>0$. Then the Dirichlet series
\begin{gather*}
f(s)=\sum_{n=1}^\infty\frac{a_n}{n^s}
\end{gather*}
converges in the region $\operatorname{Re}(s)>\sigma_0$ for some $\sigma_0>2-\delta$.
For each $n\in\mathbb Z_{>0}$, let
\begin{gather*}
H_n[f]=\det_{1\le i,j\le n}\bigl(f(i+j)\bigr).
\end{gather*}

\begin{lemma}
\label{L-est}
As $n\to\infty$, the following estimate is valid:
\begin{gather*}
\log |H_n[f]|\le -\frac12 n^2\log n+O\big(n^2\big).
\end{gather*}
\end{lemma}

\begin{proof}
Using the linearity of determinant with respect to each row and the formula for the Vandermonde determinant we have that
\begin{gather*}
H_n[f]
=\sum_{k_1=1}^\infty\dotsi\sum_{k_n=1}^\infty\frac{a_{k_1}\dotsb a_{k_n}}{k_1^2k_2^3\cdots k_n^{n+1}}\prod_{1\le i<j\le n}\big(k_i^{-1}-k_j^{-1}\big).
\end{gather*}
Now by considering one of the $n!$ possible orderings of the integers $k_1,\dots,k_n$, and using our assumption on the numbers~$a_n$ we obtain that
\begin{gather*}
|H_n[f]|
 \le n!\sum_{1\le k_1<\dots<k_n}\frac{a_{k_1}\dotsb a_{k_n}}{k_1^2k_2^3\cdots k_n^{n+1}}
\ll n!\sum_{k_1=1}^\infty\dotsi\sum_{k_n=1}^\infty\frac{1}{k_1^{1+\delta}k_2^{1+\delta}\dotsb k_n^{1+\delta}}\left(\prod_{i=1}^{n}i^{i-1}\right)^{-1}
\\
\hphantom{|H_n[f]|}{}
\ll n!\exp\left(-\frac12 n^2\log n+O\big(n^2\big)\right),
\end{gather*}
where the estimate
\begin{gather*}
\sum_{i=1}^n(i-1)\log i>\int_2^n(x-2)\log x\,{\mathrm d}x \qquad\text{for}\quad n\ge3
\end{gather*}
is invoked. Finally, taking logarithms gives the desired result.
\end{proof}

As a slight generalization of the above argument, let $a$ and $b$ be positive integers and consider the Hankel determinants
\begin{gather*}
H_n^{(a,b)}[f]=\det_{1\le i,j\le n}\bigl(f(a(i+j)+b)\bigr).
\end{gather*}
By the same steps as before, we obtain that
\begin{gather*}
|H_n^{(a,b)}[f]|\ll n!\left(\prod_{i=1}^{n}i^{ai}\right)^{-1}=n!\exp\left(-\frac a2 n^2\log n+O\big(n^2\big)\right),
\end{gather*}
and taking logarithms gives us the following result.

\begin{lemma}
\label{L-gest}
As $n\to\infty$, the following estimate is valid:
\begin{gather*}
\log |H_n^{(a,b)}[f]|\le -\frac a2 n^2\log n+O\big(n^2\big).
\end{gather*}
\end{lemma}

\section{Nonvanishing}
\label{non}

In this section, suppose that $2\le n_1<n_2<\dots<n_m<\dotsb$ is an arbitrary sequence, and assume that the corresponding zeta values
$\zeta(n_1),\zeta(n_2),\dots,\zeta(n_m),\dots$ are \emph{rational} numbers.

\begin{lemma}
\label{LW1}
The function $f(z)=\sum\limits_{m=1}^\infty \zeta(n_m)z^m$ is not rational.
\end{lemma}

\begin{proof}
Suppose that the statement of the lemma is not true. Then the quantities $\zeta(n_m)$ satisfy a linear recurrence with constant coef\/f\/icients.
By our hypothesis that the numbers~$\zeta(n_m)$ are all rational, we may assume that the coef\/f\/icients in the recurrence equation are all rational numbers,
in other words that the recurrence rule is given by
\begin{gather*}
r_0\zeta(n_m)+r_1\zeta(n_{m+1})+\dots+r_k\zeta(n_{m+k})=0
\qquad\text{for}\quad m=m_1,m_1+1,\dots,
\end{gather*}
with $r_0,\ldots, r_k\in\mathbb{Q}$ and $r_0,r_k\not= 0$.
Then, for all $m\ge m_1$, we have that
\begin{gather*}
\sum_{\ell=1}^\infty\left(\frac{r_0}{\ell^{n_m}}+\frac{r_1}{\ell^{n_{m+1}}}+\dots+\frac{r_k}{\ell^{n_{m+k}}}\right)=0,
\end{gather*}
and taking the limit as $m\to\infty$ this gives that
\begin{gather*}
r_0+r_1+\dots+r_k=0.
\end{gather*}
This in turn implies that
\begin{gather*}
\sum_{\ell=2}^\infty\left(\frac{r_0}{\ell^{n_m}}+\frac{r_1}{\ell^{n_{m+1}}}+\dots+\frac{r_k}{\ell^{n_{m+k}}}\right)=0.
\end{gather*}
Multiplying by $2^{n_m}$ and taking the limit as $m\to\infty$ (which clearly exists since the right-hand side is~$0$), we obtain that
\begin{gather*}
\lim_{m\to\infty}\sum_{\ell=2}^\infty\left(\frac{r_0}{(\ell/2)^{n_m}}+\frac{r_1}{\ell^{n_{m+1}}/2^{n_m}}+\dots+\frac{r_k}{\ell^{n_{m+k}}/2^{n_m}}\right)\\
\qquad{}=\lim_{m\to\infty}\left(r_0+\frac{r_1}{2^{n_{m+1}-n_m}}+\dots+\frac{r_k}{2^{n_{m+k}-n_m}}\right)=0.
\end{gather*}
Since $r_0\ne 0$ and $n_m<n_{m+1}<\dots<n_{m+k}$, the f\/inal equality here implies that the quantity
\begin{gather*}
r_0+\frac{r_1}{2^{n_{m+1}-n_m}}+\dots+\frac{r_k}{2^{n_{m+k}-n_m}}
\end{gather*}
is actually constant for all $m\ge m_2$ (with some $m_2\ge m_1$). Therefore we conclude that
\begin{gather*}
\frac{r_0}{2^{n_m}}+\frac{r_1}{2^{n_{m+1}}}+\dots+\frac{r_k}{2^{n_{m+k}}}=0 \qquad\text{for}\quad m\ge m_2.
\end{gather*}
Proceeding in a similar way, we deduce that for each integer $1\le\ell\le k$, there is an integer $m_k$ such that
\begin{gather*}
\frac{r_0}{\ell^{n_m}}+\frac{r_1}{\ell^{n_{m+1}}}+\dots+\frac{r_k}{\ell^{n_{m+k}}}=0 \qquad\text{for}\quad m\ge m_k.
\end{gather*}
This implies that $r_0=r_1=\dots=r_k=0$, which is a contradiction.
\end{proof}

\begin{lemma}
\label{LW2}
The Hankel determinants $\det\limits_{1\le i,j\le n}\bigl(\zeta(n_{i+j-1})\bigr)$ are not zero for infinitely many $n$.
\end{lemma}

\begin{proof}
This follows from Lemma~\ref{LW1}, together with a well known result of Kronecker (see \cite[pp.~566--567]{Kr81} or
\cite[Division~7, Problem~24]{PS76}).
\end{proof}

Note that the statement of Lemma~\ref{LW2} in the case when $n_m=am+b$ can be obtained using the argument from \cite[Lemma~2.3]{Mo09}:
in fact, the Hankel determinants are all \emph{positive} in this case.

\section{Proof of Theorem~\ref{th-main}}
\label{proof}

If there are only f\/initely many $n\in\mathbb{N}$ for which $\zeta(an+b)$ is irrational, then we can assume without loss of generality,
by replacing $b$ with $an_0+b$ for a suitable choice of $n_0$, that all of the numbers $\zeta(an+b)$ are rational.
It is clear from the def\/inition that $q_k\,|\, q_n$ for all $k\le n$, and that
\begin{gather*}
q_{n+1}q_{n+2}\dotsb q_{2n}H_n^{(a,b)}[\zeta]=\det_{1\le i,j\le n}\bigl(q_{i+n}\zeta(a(i+j)+b)\bigr)\in\mathbb Z
\qquad\text{for}\quad n\in\mathbb{N}.
\end{gather*}
By Lemma~\ref{LW2} there are inf\/initely many indices $n$ for which the determinants $H_n^{(a,b)}[\zeta]$ are nonzero, therefore
\begin{gather*}
q_{n+1}q_{n+2}\dotsb q_{2n}\cdot|H_n^{(a,b)}[\zeta]|\ge1.
\end{gather*}
If it were the case that $q_n\le C^n$ for some $C>0$ and all $n$, then we would deduce that
$|H_n^{(a,b)}[\zeta]|\ge C^{-n(3n+1)/2}$ for inf\/initely many $n$. However, for suf\/f\/iciently large $n$ this would contradict Lemma~\ref{L-gest},
which leads us to conclude that the sequence $q_n^{1/n}$ is unbounded as~$n\to\nobreak\infty$.

\section{Concluding remarks}
\label{remarks}

The same techniques used to prove Theorem~\ref{th-main} can also be applied to a much broader class of Dirichlet series.
For example, similar results can be obtained for more Dirichlet $L$-functions and even $L$-functions attached to modular forms,
by extending the argument in the proof of Lemma~\ref{LW1}. In fact, one does not even need
to restrict to values of these functions at integers. In a dif\/ferent direction, one can deal with the values at subsequences which tend to inf\/inity
faster than arithmetic progressions (for example, the sequence $\zeta(2n^2+1)$, where $n=1,2,\dots$).
The estimates from Section~\ref{est} would then become sharper, and the corresponding growth of the common denominators,
provided inf\/initely many of the $L$-values are rational, can be then shown to be faster.
Of course, as mentioned in the introduction, we expect that the nonzero zeta and $L$-function values at positive integers are always irrational (and even transcendental).
Our principal result here is only to serve as an illustration of a much deeper relationship between Hankel determinants and arithmetic.

\subsection*{Acknowledgements}

We would like to thank the anonymous referees for their healthy criticism on the earlier draft of the note.
Research of f\/irst author (Alan Haynes) was supported by EPSRC grants EP/L001462, EP/J00149X, EP/M023540.
Research of second author (Wadim Zudilin) was supported by Australian Research Council grant DP170100466.

\pdfbookmark[1]{References}{ref}
\LastPageEnding


\begin{thebibliography}{99}
\footnotesize\itemsep=0pt

\bibitem{Ap79}
Ap{\'e}ry R., Irrationalit\'e de $\zeta(2)$ et $\zeta(3)$,
  \textit{Ast\'erisque}  \textbf{61} (1979), 11--13.

\bibitem{BR01}
Ball K., Rivoal T., Irrationalit\'e d'une inf\/init\'e de valeurs de la fonction
  z\^eta aux entiers impairs, \href{http://dx.doi.org/10.1007/s002220100168}{\textit{Invent. Math.}} \textbf{146} (2001),
  193--207.

\bibitem{Kr81}
Kronecker L., Zur Theorie der Elimination einer Variabeln aus zwei
  algebraischen Gleichungen, \textit{Berl. Monatsber.} \textbf{1881} (1881),
  535--600.

\bibitem{Mo09}
Monien H., Hankel determinants of Dirichlet series, \href{http://arxiv.org/abs/0901.1883}{arXiv:0901.1883}.


\bibitem{Mo11}
Monien H., {P}ersonal communication, 2011.

\bibitem{PS76}
P{\'o}lya G., Szeg{\H{o}} G., Problems and theorems in analysis, {V}ol.~{II},
  \href{http://dx.doi.org/10.1007/978-1-4757-6292-1}{Springer-Verlag}, New York~-- Heidelberg, 1976.

\end{thebibliography}
\end{document}